\documentclass[a4paper,twoside]{article}
\usepackage{amssymb,amsfonts,amsmath,amsbsy,theorem,amscd}
\usepackage{dina42e}
\usepackage{stmaryrd}
\usepackage{extarrows}
\usepackage{bookmark,hyperref}
\usepackage{csquotes}

\usepackage{color}

\definecolor{CMUrot}{RGB}{128,18,18}
\definecolor{Gold}{RGB}{238,180,34}
\allowdisplaybreaks[4]

\newcommand{\ol}[1]{\overline{#1}}

\numberwithin{equation}{section}

\newcommand{\R}{\ensuremath{\mathbb{R}}}

\newcommand{\Rn}{\ensuremath{\mathbb{R}^2}}

\newcommand{\Om}{\ensuremath{\Omega}}

\newcommand{\N}{\ensuremath{\mathbb{N}}}

\newcommand{\dist}{\operatorname{dist}}
\newcommand{\sdist}{\operatorname{sdist}}

\newcommand{\sd}{{\, d}}
\newcommand{\supp}{\operatorname{supp}}
\newcommand{\eps}{\ensuremath{\varepsilon}}

\newcommand{\Div}{\operatorname{div}}

\newcommand{\T}{\ensuremath{\mathbb{T}}}

\newcommand{\no}{\mathbf{n}}

\def\nn{\mathbf{n}}

\newcommand{\ve}{\mathbf{v}}

\newcommand{\btau}{{\boldsymbol{\tau}}}
\newtheorem{thm}{THEOREM}[section]
\newtheorem{cor}[thm]{Corollary}
\newtheorem{lem}[thm]{Lemma}
\newtheorem{defn}[thm]{Definition}

\newtheorem{prop}[thm]{Proposition}

\newtheorem{claim*}{Claim}

\theorembodyfont{\rmfamily}
\newtheorem{rem}[thm]{Remark}

\newenvironment{proof*}[1]{{\bf Proof
#1:}}{\hspace*{\fill}\rule{1.2ex}{1.2ex}\\ }
\newenvironment{proof}{{\bf
Proof:\,}}{\hspace*{\fill}\rule{1.2ex}{1.2ex}\\ }

\newcommand{\p}{\partial}
\newcommand{\G}{\Gamma}
\def\({\left(}
\def\){\right)}

\newcommand{\tc}{\hat{c}}

\newcommand{\tr}{\hat{r}}

\newcommand{\order}{N}

\begin{document}
\begin{titlepage}
\title{Convergence of a Convective Allen-Cahn Equation to a Transport Equation}
\author{Helmut Abels}
\end{titlepage}

\maketitle
\abstract{We show convergence of solutions of a convective Allen-Cahn equation for a given smooth and divergence free velocity field to a transport equation for an evolving interface in the case when the thickness of the diffuse interface tends to zero and the mobility coefficient is proportional to the interfacial thickness. This is done for well-prepared initial data by estimating the difference of the exact and an approximate solution of the convective Allen-Cahn equation with the aid of a uniform lower bound for the linearized Allen-Cahn operator. The approximate solution is constructed with the aid of three terms from formally matched asymptotics calculations close to the interface.
	 }


         \smallskip

       {\small\noindent
}

\section{Introduction and Main Result}\label{sec:Introduction}

In applications interfaces between regions of two (or more) species or components play an important role. There are two model classes to describe them: diffuse and sharp interface models. In the latter the interface is a lower dimensional hypersurface separating two regions and the evolution is usually described by a geometric evolution equation for the interface, possibly coupled with other equations in the regions separated by the interface. In diffuse interface models{,} the interface is described  as an interfacial layer of small, but positive thickness, which will be proportional to $\eps>0$ in the following. Both for a theoretical understanding and practical applications a rigorous understanding of the relation between sharp and diffuse interface models when $\eps\to 0$ -- also called ``sharp interface limit'' -- is desirable. Despite a lot of mathematical effort throughout the last decades rigorous results on these sharp interface limits are still a challenge and open in many cases.

In this contribution we want to present a basic method to treat such sharp interface limits rigorously, which is also called ``rigorous asymptotic expansion'' and goes back to de Mottoni and Schatzman~\cite{DeMottoniSchatzman} to show convergence of the Allen{-}Cahn equation to the mean curvature flow equation as long as the latter possesses a smooth solution. Afterwards this method was extended by Alikakos, Bates, Chen \cite{AlikakosLimitCH} to the Cahn-Hilliard equation, by Caginalp and Chen \cite{CaginalpChen} to the phase field equation, by Chen, Hilhorst and Logak~\cite{ChenHilhorstLogak} to the volume preserving Allen-Cahn equation, by Abels and Moser~\cite{AbelsMoserClose90} to the Allen-Cahn equation with contact angle close to $90^\circ$, by Moser~\cite{MoserAsy,MoserAdv} to the vectorial Allen-Cahn equation with $90^\circ$ contact angles at the boundary, by  Abels, Liu \cite{StokesAllenCahn} to a Stokes/Allen-Cahn system, and recently by Abels, Fei \cite{AbelsFei} and Abels, Moser, Fei~\cite{AbelsFeiMoser} for a Navier-Stokes/Allen Cahn system with constant mobility in different cases. The method consists of two steps. In the first step an approximate solution to the equations for the diffuse interface model is constructed, which is usually based on formulae derived by formally matched asymptotic calculations. In the second step the difference between the approximate and exact solution is estimated with the aid of energy-type estimates and suitable lower bounds for the linearized operator. 
We apply this method to a nowadays relative simple model case, which is the following convective Allen-Cahn equation:
\begin{align}
  \partial_t c_\eps + \ve \cdot \nabla c_\eps & = m_\eps\left( \Delta c_\eps -\eps^{-2} f'(c_\eps)\right) && \mbox{in }\Omega \times (0,T_0), \label{ConvAC1}\\
  c_\eps|_{\partial\Omega} & = -1 && \mbox{on } \partial \Omega \times (0,T_0), \label{ConvAC3}\\
  \left.c_\eps \right|_{t=0}&= c_{0,\eps} &&\mbox{in } \Omega . \label{ConvAC4}
\end{align}
Here  $\ve\colon \Omega \times [0,T) \to \mathbb{R}^2$ is a given smooth velocity field with $\no \cdot \ve|_{\partial\Omega}=0$ and $\Div \ve=0$ in $\ol{\Omega}\times [0,T_0]$. Moreover,  $c_\eps \colon \Omega \times [0,T) \to \mathbb{R}$ is a function, which is close to $\pm 1$ away from a diffuse interface of thickness $\eps>0$ and can {for example} describe the difference of volume fractions of two fluids. Furthermore, $m_\eps>0$ is a (constant) mobility coefficient and $f\colon \R\to \R$ is a suitable double well potential with global minima at $\pm 1$. More precisely, we assume that $f$ satisfies 
\begin{equation*}
  f'(\pm 1)=0, \quad f''(\pm 1)>0,\quad f(s)=f(-s)>0 \quad \text{for all }s\in (-1,1).
\end{equation*}
as in Abels and Liu~\cite{StokesAllenCahn}. We remark that under this assumption there is a unique solution of
\begin{align}\label{eq:OptProfile}
	-\theta_0''+f'(\theta_0)=0\quad\text{ in }\R, \qquad
	\lim_{\rho\to\pm \infty}\theta_0(\rho)=\pm 1, \qquad \theta_0(0)=0,
\end{align}
which is called \emph{optimal profile} and plays an important role in the following since it gives the shape of the diffuse interface in normal direction to leading order, cf.\ \eqref{eq:Expansion} below.
Moreover, we note that for every $m\in\N_0$ there is some $C_m>0$ such that
\begin{equation}\label{eq:decayopti}
  |\p_\rho^m(\theta_0(\rho)\mp 1)|\leq C_m e^{-\alpha|\rho|}\quad\text{for all } \rho\in\R  \text{ with }\rho \gtrless 0,
\end{equation}
where $\alpha=\min(\sqrt{f''(-1)}, \sqrt{f''(1)})$.

Throughout the manuscript
$\Omega\subseteq\Rn$ is assumed to be a bounded domain with smooth boundary. We restrict ourselves to the two-dimensional case for simplicity. But the following result can be extended to arbitrary space dimensions provided additional terms in the construction of the approximate solution $c_A$ in Section~\ref{sec:Approx} with higher powers of $\eps$ are added to obtain a sufficiently high power $N$ of $\eps$ for the approximation error. In the present case we can work with $N=2$.   

We remark that the convective Allen-Cahn equation \eqref{ConvAC1} is part of the following diffuse interface model for the two-phase flow of two incompressible Newtonian fluids with phase transition
\begin{align}
\partial_t \ve_\eps + \ve_\eps \cdot \nabla \ve_\eps - \Div(\nu(c_\eps) D\ve_\eps) + \nabla p^\eps & = -\eps \Div(\nabla c_\eps \otimes \nabla c_\eps), \label{NSAC1}\\
\Div\, \ve_\eps & = 0,  \\\label{NSAC3}
\partial_t c_\eps + \ve_\eps \cdot \nabla c_\eps & = m_\eps\left( \Delta c_\eps -\eps^{-2} f'(c_\eps)\right)
\end{align}
in $\Omega\times (0,T)$.
Here $\ve_\eps\colon \Omega\times [0,T)\to \Rn$ is the velocity of the mixture, $D \ve_\eps = \frac{1}{2} (\nabla \ve_\eps + (\nabla \ve_\eps)^T)$, $p^\eps\colon \Omega \times [0,T)\to \R$ is the pressure, and $\nu(c_\eps)>0$ is the viscosity of the mixture. We refer to \cite{AbelsFei}, \cite{AbelsFeiMoser}, or \cite{AbelsFischerMoser} for further comments, references for the system and results on the sharp interface limit.

For the convective Allen-Cahn equation \eqref{ConvAC1} as well as for the Navier-Stokes/Allen-Cahn system \eqref{NSAC1}-\eqref{NSAC3} the convergence as $\eps\to 0$ depends on the precise dependence of the mobility coefficient $m_\eps>0$ on $\eps$. For the convective Allen-Cahn equation formal convergence of solutions \eqref{NSAC1}-\eqref{NSAC3} was shown by the author in \cite{AbelsConvectiveAC} to a convective mean curvature flow equation
\begin{equation}\label{eq:MCF}
  V_{\Gamma_t} -\no_{\Gamma_t}\cdot \ve|_{\Gamma_t} = m_0 H_{\Gamma_t}
\end{equation}
in the that case $m_\eps = m_0>0$ is independent of $\eps$ and to a transport equation
\begin{equation}\label{eq:transport}
  V_{\Gamma_t} =\no_{\Gamma_t}\cdot \ve|_{\Gamma_t} 
\end{equation}
in the case $m_\eps= m_0 \eps$. This is done with the aid of formally matched asymptotic calculations. Here $V_{\Gamma_t}$ and $H_{\Gamma_t}$ denote the normal velocity and (mean) curvature of $\Gamma_t$, respectively, taken with respect to the interior normal on $\no_{\Gamma_t}$. Moreover, in the case $m_\eps = m_0 \eps^k$ for some $k>2$ non-convergence is shown in a certain sense.
A first rigorous convergence result for a Stokes/Allen-Cahn system (i.e., \eqref{NSAC1}-\eqref{NSAC3} without the terms $\partial_t \ve_\eps + \ve_\eps \cdot \nabla \ve_\eps$)  was shown by Liu and the author in \cite{StokesAllenCahn} in the case $m_\eps = m_0$ with constant viscosity. This result was extended in \cite{AbelsFei} and \cite{AbelsFeiMoser} to the case of the Navier-Stokes/Allen-Cahn system with non-constant viscosity with $m_\eps= m_0$ and $m_\eps= m_0 \sqrt{\eps}$. In these contributions the rigorous asymptotic expansion method is used and a two-dimensional situation is considered. A different method based on a relative entropy method developed by Fischer, Laux and Simon~\cite{FischerLauxSimon_AC_MCF} was used to show convergence of \eqref{NSAC1}-\eqref{NSAC3} by Hensel and Liu in \cite{HenselLiu} in the case $m_\eps = m_0$ and by Abels, Fischer and Moser in the case $m_\eps = m_0 \eps^k$ with $k\in (0,2)$ with constant viscosity and in two and three space dimensions. We refer to \cite{AbelsFeiMoser} for a comparison of the method and further references. 

It is the goal of the present contribution to prove convergence of \eqref{ConvAC1}-\eqref{ConvAC4} to \eqref{eq:transport} for well-prepared initial data in the case $m_\eps = m_0 \eps$. We note that an analogous convergence result to solutions of \eqref{eq:MCF} in the case $m_\eps = m_0$ and as long as \eqref{eq:MCF} possesses a smooth solution can be obtained from the arguments in \cite{StokesAllenCahn} by simply neglecting the coupling to the Stokes equations and taking $\ve_\eps\equiv \ve$ smooth and independent of $\eps$ (and $\ve_A\equiv \ve$). Our main result is as follows:
\begin{thm}\label{thm:main}
  Let $N=2$, $T_0>0$, $\Omega\subseteq \R^2$ be a bounded and smooth domain, $\ve\colon \ol{\Omega}\times [0,T_0]\to \R^2$ be smooth with $\Div \ve=0$ and $\ve|_{\partial\Omega} =0$, $\Gamma_0\subseteq \Omega$ be a smooth and closed Jordan curve, and $m_\eps=m_0\eps$ for all $\eps>0$ and some $m_0>0$. Moreover, let $(\Gamma_t)_{t\in [0,T_0]}$ be the curves satisfying \eqref{eq:transport} for all $t\in (0,T_0)$ and $\Gamma_t|_{t=0}=\Gamma_0$.  Then there exist $c_A=c_A(\eps)\colon \ol{\Omega}\times [0,T_0]\to \R$ for $\eps\in(0,1]$, uniformly bounded, such that the following holds: 

If $c_\eps\colon \ol{\Omega}\times [0,T_0]\to \R$, $\eps\in (0,1]$, are solutions of \eqref{ConvAC1}-\eqref{ConvAC3} with initial values  $c_{0,\eps}$ such that
\begin{equation}\label{initial assumption}
	\|c_{0,\eps}-c_A|_{t=0}\|_{L^2(\Omega)}\leq C\eps^{\order+\frac12}
\end{equation}
for all $\eps\in (0,1]$ and some $C>0$,
then there are some $\eps_0 \in (0,1]$ and $R>0$ such that for all $\eps \in (0,\eps_0]$ it holds
\begin{align}
		\|c_\eps-c_A\|_{L^\infty(0,T_0;L^2(\Omega))}+\eps^{\frac12}\|\nabla (c_\eps -c_A)\|_{L^2((\Omega\times (0,T_0))\setminus\Gamma(\delta))}  &\leq R\eps^{\order+\frac12},\label{eq_convC1}\\
		\eps^{\frac12} \|\nabla_{\btau}(c_\eps -c_A)\|_{L^2((\Omega\times(0,T_0))\cap \Gamma(2\delta))}+ \eps \|\nabla (c_\eps -c_A)\|_{L^2((\Omega\times(0,T_0))\cap \Gamma(2\delta))} &\leq R\eps^{\order+\frac12},\label{eq_convC2}
\end{align}
where $\Gamma(\tilde{\delta})$ are  tubular neighborhoods of $\Gamma= \bigcup_{t\in [0,T_0]}\Gamma_t\times \{t\}$ for $\tilde{\delta}\in[0,3\delta]$, $\delta>0$ small and $\nabla_{\btau}$ is a suitable tangential gradient, see Section \ref{sec:Prelim} below. Moreover, 
\begin{alignat}{2}\label{eq:Expansion}
  c_A&=\zeta(d_{\Gamma_t})\theta_0(\rho_\eps) + (1-\zeta(d_\Gamma))(\chi_{\Omega^+}-\chi_{\Omega^-}) + O(\eps)&\quad&\text{ in }L^\infty((0,T_0)\times\Om),
\end{alignat}
where $\eps\rho_\eps=d_{\Gamma}+O(\eps)$ in $\Gamma(3\delta)$, $d_\Gamma(\cdot,t)$ is the signed distance to $\Gamma_t$ and $\zeta\colon \R\to [0,1]$ is smooth such that $\supp \zeta \subseteq [-2\delta,2\delta]$ and $\zeta\equiv 1$ on $[-\delta,\delta]$. In particular, 
\[
\lim_{\eps\to 0} c_A=\pm 1\quad\text{ uniformly on compact subsets of } \Omega^\pm=\bigcup_{t\in [0,T_0]}\Omega_t^\pm\times \{t\},
\]
where $\Omega^+_t$ and $\Omega^-_t$ are the interior and exterior of $\Gamma_t$ (in $\Omega$), respectively.
\end{thm}
\begin{rem}
  We note that the solution $(\Gamma_t)_{t\in [0,T_0]}$ of \eqref{eq:transport} with initial curve $\Gamma_0$ are simply obtained by $\Gamma_t = X_t(\Gamma_0)$, where $X_t\colon \ol{\Omega}\times [0,T_0]\to \ol{\Omega}$ is determined as solution of the initial value problem
  \begin{alignat*}{2}
    \frac{d}{dt} X_t(\xi) &= \ve (X_t(\xi),t) &\qquad &\text{for all }t\in[0,T_0],\xi\in \ol{\Omega},\\
    X_0 (\xi) &= \xi&\qquad &\text{for all }\xi\in \ol{\Omega},
  \end{alignat*}
\end{rem}

Finally, we note that convergence results for the Allen-Cahn equation to the mean curvature flow equation globally in time using different concepts can e.g.\ be found in
\cite{BarlesDaLio,BarlesSouganidis,EvansSonerSouganidis,HenselLaux,Ilmanen,Kagaya,KatsoulakisKR,LauxSimon_BV,MizunoTonegawa}.

The structure of the contribution is as follows: In Section~\ref{sec:Prelim} we collect some preliminary results and introduce some notation. In Section~\ref{sec:Approx} we perform the first step in the rigorous expansion method, i.e., we construct an approximate solution of \eqref{ConvAC1}-\eqref{ConvAC4} up to an error of order $O(\eps^{2+\frac12})$ in $L^2(\Omega\times (0,T_0))$. Then in Section~\ref{sec:Main} we estimate the error between the approximate and exact solution of \eqref{ConvAC1}-\eqref{ConvAC4} with the aid of a spectral lower bound for the linearized Allen-Cahn operator. This proves our main result. 

\section{Preliminaries}\label{sec:Prelim}

We use the notation of \cite{StokesAllenCahn}. In particular, $\N$ denotes the set of natural numbers (without $0$), $\N_0=\N\cup\{0\}$, and $L^p(M)$, $1\leq p\leq\infty$ denotes the standard $L^p$-space on a measurable subset $M\subseteq \R^N$ and $L^p(M;X)$ its $X$-valued variant for a Banach space $X$. The standard $L^2$-Sobolev space on an open set $U$ of order $m\in\N_0$ is denoted by $H^m(U)$.

In the following we recall some basic notation and important preliminary results.
We parameterize $(\Gamma_t)_{t\in[0,T_0]}$ with the aid of a family of smooth diffeomorphisms $X_0\colon \T^1\times [0,T_0]\to \Om$ such that $\partial_s X_0(s,t)\neq 0$ for all $s\in\T^1$, $t\in [0,T_0]$ and define
 \begin{equation*}
\btau(s,t)= \frac{\partial_s X_0(s,t)}{|\partial_s X_0(s,t)|}\quad \text{and}\quad \no(s,t)=
\begin{pmatrix}
  0 & -1\\
1 & 0
\end{pmatrix}
\btau(s,t)\quad \text{for all }(s,t)\in \T^1\times [0,T_0].
\end{equation*}
The orientation of $\Gamma_t$ (induced by $X_0(\cdot,t)$) is chosen such that $\no(s,t)$ is the interior normal with respect to $\Omega^+_t$. Furthermore, we denote
\begin{equation}\label{eq:1.58}
\no_{\Gamma_t}(x):= \no (s,t)\quad\text{for all}~ x=X_0(s,t)\in \Gamma_t.
\end{equation}
and by $V_{\Gamma_t}$ and $H_{\Gamma_t}$ as the normal velocity and (mean) curvature of $\Gamma_t$ (with respect to $\no_{\Gamma_t}$). Finally, 
we define
 \begin{equation}\label{eq:1.57}
   V(s,t)= V_{\Gamma_t}(X_0(s,t)),\quad H(s,t)= H_{\Gamma_t}(X_0(s,t))\quad \text{for all }s\in\T^1, t\in [0,T_0].
 \end{equation}
Then it holds
\begin{equation*}
  V_{\Gamma_t}(X_0(s,t))=V(s,t)= \partial_t X_0(s,t)\cdot \no(s,t)\qquad \text{for all }(s,t)\in \T^1\times [0,T_0].
\end{equation*}

It is well-known that for $\delta>0$ sufficiently small,  the orthogonal projection $P_{\Gamma_t}(x)$ of all
\begin{equation*}
x\in \Gamma_t(3\delta) =\{y\in \Omega: \dist(y,\Gamma_t)<3\delta\}
\end{equation*}
is well-defined and smooth. We choose $\delta>0$ so small that $\dist(\partial\Omega,\Gamma_t)>3\delta$ for every $t\in [0,T_0]$. Then for every $x\in\Gamma_t(3\delta)$ we have
\begin{equation*}
x=P_{\Gamma_t}(x)+r\no_{\Gamma_t}(P_{\Gamma_t}(x)),
\end{equation*}
 where $r=\sdist(\Gamma_t,x)$. Throughout this contribution
\begin{equation*}
  d_{\G}(x,t):=\sdist (\Gamma_t,x)=
  \begin{cases}
    \dist(\Omega^-_t,x) &\text{if } x\not \in {\Omega^-_t},\\
    -\dist(\Omega^+_t,x) &\text{if } x \in {\Omega^-_t}
  \end{cases}
\end{equation*}
is the signed distance function to $\Gamma_t$ for $t\in[0,T_0]$.
For the following we denote for $\delta'\in (0,3\delta]$
\begin{equation*}
  \Gamma(\delta') =\bigcup_{t\in [0,T_0]} \Gamma_t(\delta') \times\{t\}.
\end{equation*}

We introduce new coordinates in  $\Gamma(3\delta)$ using
\begin{equation*}
  X\colon  (-3\delta, 3\delta)\times \T^1 \times [0,T_0]\to \Gamma(3\delta)~\text{by}~  X(r,s,t):= X_0(s,t)+r\no(s,t),
\end{equation*}
where
\begin{equation}\label{eq:1.42}
  r=\sdist(\Gamma_t,x), \qquad s= X_{0}^{-1}(P_{\Gamma_t}(x),t)=: S(x,t).
\end{equation}
Then
\begin{equation}\label{eq:1.26}
  \nabla d_{\G}(x,t)=\no_{\Gamma_t} (P_{\Gamma_t}(x)),~  \partial_t d_{\G}(x,t)=-V_{\Gamma_t} (P_{\Gamma_t}(x)),~\Delta d_\Gamma(q,t)=-H_{\Gamma_t}(q)
\end{equation}
for all $(x,t)\in \Gamma(3\delta)$, $q\in\Gamma_t$, $t\in [0,T_0]$, resp., cf.\ Chen et al.~\cite[Section~4.1]{ChenHilhorstLogak}.
 Furthermore, we define
\begin{equation}\label{eq:1.50}
\partial_{\btau} u(x,t):= \btau(S(x,t),t)\nabla_x u(x,t),\quad   \nabla_\btau u(x,t):= \partial_{\btau} u(x,t)\btau(S(x,t),t)\quad
 \end{equation}
for all $(x,t)\in \Gamma(3\delta)$. 

For given $\tilde{\phi}\colon (-3\delta,3\delta)\times \T^1\times [0,T_0]\to \R$ we associate a function $\phi\colon \Gamma(3\delta)\to \R$ by the relation
 \begin{equation*}
   \phi(x,t)=\tilde{\phi}(d_{\G}(x,t),S(x,t),t)\qquad \text{for all }(x,t)\in \Gamma(3\delta)
 \end{equation*}
 or equivalently
 \begin{equation*}
   \phi(X_0(s,t)+r\no(s,t),t)=\tilde{\phi}(r,s,t)\qquad \text{for all }r\in (-3\delta,3\delta),s\in \T^1, t\in [0,T_0].
 \end{equation*}
 Then \eqref{eq:1.26} and the chain rule yield for sufficiently smooth $\tilde{\phi}$
\begin{equation}\label{Prelim:1.13}
  \begin{split}
    \partial_t \phi(x,t) &= -V_{\Gamma_t} (P_{\Gamma_t}(x)) \partial_r\tilde{\phi}(r,s,t) + \partial_{t}^\Gamma \tilde{\phi}(r,s,t) \\
  \nabla \phi(x,t) &= \no_{\Gamma_t} (P_{\Gamma_t}(x)) \partial_r\tilde{\phi}(r,s,t) + \nabla^ \Gamma  \tilde{\phi}(r,s,t) \\
 \Delta \phi(x,t) &= \partial_r^2\tilde{\phi}(r,s,t) + \Delta d_{\G_t}(x) \partial_r\tilde{\phi}(r,s,t) +  \Delta^{\Gamma} \tilde{\phi}(r,s,t),
  \end{split}
\end{equation}
where $r,s$ are as in \eqref{eq:1.42} and we use the definitions
\begin{equation}\label{Prelim:1.12}
  \begin{split}
    \partial_{t}^\Gamma \tilde{\phi}(r,s,t) &:= \partial_t \tilde{\phi}(r,s,t) + \partial_t S(x,t)\partial_s \tilde{\phi}(r,s,t),\\
\nabla^{\Gamma} \tilde{\phi}(r,s,t) &:= (\nabla S)(x,t)\partial_s \tilde{\phi}(r,s,t),\\
\Delta^{\Gamma} \tilde{\phi}(r,s,t) &:= |(\nabla S)(x,t)|^2\partial_s^2 \tilde{\phi}(r,s,t)+(\Delta S)(x,t)\partial_s \tilde{\phi}(r,s,t),
  \end{split}
\end{equation}
 cf.~\cite[Section~4.1]{ChenHilhorstLogak} for more details.
We note that, if $g\colon \T^1\times [0,T_0]\to \R$ depends only on $(s,t)$, then $\nabla^\G g$ is a function of $(r,s,t)$:
\begin{equation}\label{eq:1.46}
  \nabla^\G g(r,s,t)=(\nabla S)(x,t)\partial_s g(s,t), \qquad \text{where }x=X(r,s,t)
\end{equation}
and analogously for $\Delta^\G h$ and $\p_t^\G h$.
The restrictions of these operators to $r=0$ are denoted by
\begin{equation}\label{eq:1.27}
\begin{split}
  (\nabla_\G h)(s,t):=(\nabla^\Gamma h)(0,s,t),\\
  (\Delta_\G h)(s,t):=(\Delta^\G h)(0,s,t),\\
  (D_t h)(s,t):=(\p_t^\G h)(0,s,t).
\end{split}
\end{equation}

Finally, we use the notation
\begin{alignat*}{2}
 (X_0^\ast u)(s,t)&:= u(X_0(s,t),t) &\qquad& \text{for all }s\in\T^1,t\in[0,T_0],\\
  (X_0^{\ast,-1} v)(p,t) & := v(X_0^{-1}(p,t),t) &\qquad& \text{for all }p\in\Gamma_t, t\in [0,T_0]
\end{alignat*}
 for $u\colon \bigcup_{t\in[0,T_0]}\Gamma_t\times \{t\}\to \R^N$  and $v\colon \T^{1}\times [0,T_0]\to \R^N$  for some $N\in\N$.
\section{Approximate Solutions}\label{sec:Approx}
\subsection{The Stretched Variable and Remainder Estimates}

The construction of the approximate solution will be based on a suitable scaling in normal direction proportional to the diffuse interface ``thickness'' $\eps$. To the end we use as in \cite{StokesAllenCahn} the so-called stretched variable $\rho$ defined by
\begin{equation}\label{eq:stretched}
  {\rho\equiv\rho_\eps(x,t)=\tfrac{d_\G(x,t)}{\eps}-h_1(S(x,t),t)-\eps h_2(S(x,t),t)},
\end{equation}
where $h_1,h_2\colon \T^1\times [0,T_0]\to \R$ are functions, which will determined in Section~\ref{subsec:Construction} suitably. Roughly speaking they are used to obtain that the set $\{\rho_\eps(x,t)=0\}$ is a sufficiently good approximation of $\{c_\eps(x,t)=0\}$. Moreover, we denote $h_\eps:=h_1+\eps h_2$.

For the construction of the approximate solution we will use:
\begin{lem}
Let ${\hat{w}}\colon \R\times \T^1\times [0,T_0]\to \R$ be sufficiently smooth and let
\begin{equation*}
w(x,t)={\hat{w}}\(\tfrac{d_\G(x,t)}\eps-h_\eps(S(x,t),t),S(x,t),t\) \quad \text{for all }(x,t)\in \Gamma(2\delta).
\end{equation*}
Then for each $\eps>0$
  \begin{equation}\label{eq:formula1}
  \begin{split}
    \p_t w(x,t)=&-\(\tfrac{V_{\Gamma_t} (P_{\Gamma_t}(x))}\eps +\p_t^\Gamma h_\eps(r,s,t)\)\p_\rho {\hat{w}}(\rho,s,t) +\p_t^\Gamma {\hat{w}}(r,\rho,s,t)\\
    \nabla w(x,t)=& \(\tfrac{\no_{\Gamma_t} (P_{\Gamma_t}(x))}  \eps -\nabla^\Gamma h_\eps(r,s,t)\)\p_\rho{\hat{w}}(\rho,s,t) +\nabla^\Gamma{\hat{w}}(r,\rho,s,t)\\
    \Delta w(x,t)=& (\eps^{-2}+|\nabla^\Gamma h_\eps(r,s,t)|^2) \p^2_\rho{\hat{w}}(\rho,s,t)\\
    &+\(\eps^{-1}\Delta d_\Gamma (x,t) -\Delta^\Gamma h_\eps(r,s,t)\)\p_\rho{\hat{w}}(\rho,s,t)\\
    &-2\nabla^\Gamma h_\eps(r,s,t)\cdot\nabla^\Gamma \p_\rho{\hat{w}}(r,\rho,s,t)+\Delta^\Gamma {\hat{w}}(r,\rho,s,t),
  \end{split}
\end{equation}
where $\rho$ is as in \eqref{eq:stretched} and $(r,s)$ is understood via \eqref{eq:1.42}.
\end{lem}
\begin{proof}
This follows from the chain rule and \eqref{Prelim:1.13} in a straight forward manner, cf.\ \cite[Section~4.2]{ChenHilhorstLogak}:  
\end{proof}

To treat remainder terms  the following set is useful:
\begin{defn}\label{defn:1.15}
  For any $k\in \R$ and $\alpha>0$,  $\mathcal{R}_{k,\alpha}$ denotes the vector space of all families of continuous functions $\tr_\eps\colon \R\times \Gamma(2\delta) \to \R$, $\eps\in (0,1)$, which are continuously differentiable with respect to $\no_{\Gamma_t}$ for all $t\in [0,T_0]$ such that
  \begin{equation}\label{eq:EstimRkalpha}
    |\partial_{\no_{\Gamma_t}}^j \tr_\eps(\rho,x,t)|\leq Ce^{-\alpha |\rho|}\eps^k\qquad \text{for all }\rho\in \R,(x,t)\in\Gamma(2\delta), j=0,1, \eps \in (0,1)
  \end{equation}
for some $C>0$ independent of $\rho\in \R,(x,t)\in\Gamma(2\delta)$, $\eps\in (0,1)$. Moreover,   $\mathcal{R}_{k,\alpha}$ is equipped  with the norm
\begin{equation*}
  \|(\tr_\eps)_{\eps\in (0,1)}\|_{\mathcal{R}_{k,\alpha}}= \sup_{\eps \in (0,1), (x,t)\in\Gamma(2\delta),\rho \in\R,j=0,1} |\partial_{\no_{\Gamma_t}}^j \tr_\eps(\rho,x,t)|e^{\alpha |\rho|}\eps^{-k}.
\end{equation*}
Finally, $\mathcal{R}_{k,\alpha}^0$ is the subspace of all $(\tr_\eps)_{\eps\in (0,1)}\in \mathcal{R}_{k,\alpha}$ such that
\begin{equation}\label{eq:RemainderVanish}
  \tr_\eps(\rho,x,t)= 0 \qquad \text{for all }\rho \in\R, x\in\Gamma_t, t\in [0,T_0].
\end{equation}
\end{defn}

The following corollary will be used to estimate the remainder terms in the equation for the approximate solution.
\begin{cor}\label{cor:repsEstim2}
  Let $k\in\R$, $\alpha>0$, $h_\eps \colon \T^1\times [0,T_0] \to \R$ such that
  \begin{equation*}
    M:= \sup_{{0<\eps < 1},(s,t)\in\T^1\times [0,T_0]} |h_\eps(s,t)|<\infty
  \end{equation*}
  and
  \begin{equation*}
    r_\eps (x,t):= \tr_\eps\left( \frac{d_{\G}(x,t)}\eps - h_\eps(S(x,t),t), x,t\right)\qquad \text{for all }(x,t)\in\Gamma(2\delta),\eps\in (0,1),
  \end{equation*}
  where $(\tr_\eps)_{0<\eps<1}\in \mathcal{R}_{k,\alpha}$ for some $\alpha>0$, $k\in\R$ and let $j=1$ if even $(\tr_\eps)_{0<\eps<1}\in \mathcal{R}_{k,\alpha}^0$ and $j=0$ else.
  Then there is some $C>0$ such that
  \begin{equation*}
    \left\| {a(P_{\Gamma_t}(\cdot))r_\eps} \right\|_{L^2(\G_t( 2\delta))} \leq C (1+M)^j\eps^{\frac 12+k+j} \|a\|_{L^2(\Gamma_t)}
  \end{equation*}
  uniformly for all $a\in L^2(\Gamma_t)$, $t\in [0,T_0]$, and $\eps\in (0,1)$.
\end{cor}
\begin{proof}
  This is a special case of \cite[Corollary~2.7]{StokesAllenCahn}.
\end{proof}


\subsection{Construction of the Approximate Solution}\label{subsec:Construction}

We construct the approximate solution in the form
\begin{alignat}{2}\label{eq:cA}
    c_{A}(x,t)&=\zeta(d_\Gamma(x,t){)} c^{in}(x,t)+(1-\zeta(d_\Gamma(x,t)))\left(
  \chi_{\Omega^+}(x,t)-\chi_{\Omega^-}(x,t) \right)
\end{alignat}
for all $x\in \Om$, $t\in [0,T_0]$,
where $\zeta\in C^\infty(\R)$ satisfies
\begin{equation}\label{eq:1.34}
  \zeta(s)=1~\text{if}~|s|\leq\delta; ~\zeta(s)=0~\text{if}~|s|\geq 2\delta;~ 0\leq  -s\zeta'(s) \leq 4~\text{if}~ \delta\leq |s|\leq 2\delta.
\end{equation}
We construct $c^{in}\colon \Gamma(3\delta)\to \R$ by an ``inner expansion'' and it is of the form
\begin{equation}\label{eq:innerexpan}
\begin{split}
  c^{in} (x,t)=\tc^{in}(\rho,s,t)=&\, \theta_0(\rho) + \eps \tc_1(\rho,S(x,t),t) + \eps^2 \tc_2\(\rho,S(x,t),t\)\\
  =:&\, c_0^{in}(x,t)+\eps c_1^{in}(x,t)+\eps^2c_2^{in}(x,t)\quad \text{for all }(x,t)\in \Gamma(3\delta),
\end{split}
\end{equation}
 where $\rho$ is defined as in \eqref{eq:stretched}.
 Here and in the following $x$ and  $(\rho,s)$ are related by \eqref{eq:stretched} and $s=S(x,t)$ if both variables appear. The functions $\tc_j \colon \R \times \T^1 \times [0,T_0]\to \R$ and $h_j\colon \T^1\times [0,T_0]\to \R$, $j=1,2$, will be determined in the following to obtain an approximate solution of order $O(\eps^{2+\frac12})$ in $L^2(\Omega\times (0,T_0))$. We note that in the present case the ``outer expansion'', i.e., the values of $c_A$ in $\Omega^\pm \setminus \Gamma(3\delta)$, is simply $\pm 1$.

 We note that
 \begin{equation*}
  (\nabla S)(x,t)\cdot \no(S(x,t),t)=0 \qquad \text{for all }(x,t)\in \Gamma (3\delta),
\end{equation*}
which is obtained by differentiating
\begin{equation*}
  S(X_0(s,t)+r\no (s,t))=s\qquad \text{for all }s\in \mathbb{T}^1,t\in [0,T_0], r\in (-3\delta,3\delta)
\end{equation*}
with respect to $r$ at $r=0${.}
 Thus we obtain from \eqref{eq:1.46} 
\begin{equation*}
  \left|\tfrac{\no(s,t)}\eps -{(\nabla^\Gamma h_\eps)(r,s,t)} \right|^2
= \frac{1}{\eps^2} + |\nabla^\Gamma  h_\eps(r,s,t) |^2
\end{equation*}
for all $r\in(-3\delta,3\delta), s=S(x,t)\in \mathbb{T}^1, t\in [0,T_0]$.
Hence \eqref{eq:formula1}, \eqref{eq:1.58}, and \eqref{eq:1.57} yield
\begin{equation}\label{eq:expanca1}
 \begin{split}
    \partial_t c_0^{in}(x,t)&= \theta_0'(\rho)\( -\tfrac{V(s,t)}\eps
    - \p_t^\Gamma  h_\eps(r,s,t)
    \){,}\\
  \nabla c_0^{in}(x,t)&= \theta_0'(\rho)\left( \tfrac{\no(s,t)}\eps  - \nabla^\Gamma h_\eps(r,s,t) \right){,}\\
 \eps\Delta c_0^{in}(x,t)&=\theta_0''(\rho) \(\tfrac1{\eps} +\eps|\nabla^\Gamma h_\eps(r,s,t)|^2\) + \theta_0'(\rho)\left( \Delta d_{\Gamma}(x,t) -\eps\Delta^\Gamma h_\eps(r,s,t)\right)
\end{split}
\end{equation}
and similarly, for $k=1,2$,
\begin{equation}\label{eq:expanca2}
  \begin{split}
    \partial_t c_k^{in}(x,t)&= \( -\tfrac{V(s,t)}\eps
    - \p_t^\Gamma  h_\eps(r,s,t)
    \)\partial_\rho \tc_k(\rho,s,t) +  \p_t^\Gamma \tc_k(r,\rho,s,t){,}\\
  \nabla c_k^{in}(x,t)&= \left( \tfrac{\no}\eps  - \nabla^\Gamma h_\eps(r,s,t) \right)\partial_\rho \tc_k(\rho,s,t) +   \nabla^\Gamma \tc_k(r,\rho,s,t){,}\\
  \eps\Delta c_k^{in}(x,t)&=  (\eps^{-1}+\eps|\nabla^\Gamma h_\eps(r,s,t)|^2)\partial_\rho^2\tc_k(\rho,s,t)\\
  &\quad + \left( \Delta d_{\Gamma}(x,t) -  \eps\Delta^\Gamma h_\eps(r,s,t)  \right)\partial_\rho \tc_k(\rho,s,t)\\
& \quad -2\eps \nabla^\Gamma h_\eps(r,s,t)\cdot  \nabla^\Gamma \partial_\rho \tc_k(r,\rho,s,t) + \eps\Delta^\Gamma  \tc_k(r,\rho,s,t){.}
  \end{split}.
\end{equation}
By a Taylor expansion we obtain
\begin{align*}
  \frac1{\eps} f'(c^{in}(x,t))&= \frac1{\eps} f'(\theta_0(\rho)) +f''(\theta_0(\rho)) \tc_1(\rho,s,t)+\eps f''(\theta_0(\rho)) \tc_2(\rho,s,t)+ \eps^2\mathfrak{R},
\end{align*}
where
\begin{equation*}
  \mathfrak{R}   = \frac12f'''\(\theta_0(\rho)+ \xi(\rho,s,t)({\eps}\tc_1+\eps^{{2}}\tc_2)(\rho,s,t)\)\left(\tc_1(\rho,s,t)+\eps\tc_2(\rho,s,t)\right)^2
\end{equation*}
for some $\xi(\rho,s,t)\in [0,1]$.
This together with  \eqref{eq:expanca1} and \eqref{eq:expanca2} lead to
\begin{alignat}{1}\nonumber
    &\p_t c^{in}(x,t)+\ve(x,t)\cdot \nabla c^{in}(x,t)-\eps\Delta c^{in}(x,t)+\frac1{\eps} f'(c^{in}(x,t))\\\nonumber
  &= \frac {-V(s,t)+\ve(x,t)\cdot \no(s,t)}\eps \theta_0'+ (-V(s,t)+\ve(x,t)\cdot \no(s,t)) \partial_\rho \tc_1\\\nonumber
  & \quad +(-\p_t^\Gamma h_1-\ve\cdot \nabla^\Gamma h_1-\Delta d_{\Gamma})\theta_0'-\p_\rho^2 \tc_1+f''(\theta_0) \tc_1-\eps|\nabla_\Gamma h_1|^2\theta_0''\\
  &\quad +\eps(-\p_t^\Gamma h_2-\ve\cdot \nabla^\Gamma h_2+\Delta^\Gamma h_1)\theta_0'-\eps\big(\p_\rho^2\tc_2+f''(\theta_0)\tc_2\big)\label{eq:monst1}
  +R_1^0 +R_2{,} 
\end{alignat}
where $R_1^0\in \mathcal{R}_{1,\alpha}^0$ and $R_2\in \mathcal{R}_{2,\alpha}$. Here and in the following all functions as e.g.\
$
 \theta_0',  \tc_1, h_1, \nabla^\Gamma \tc_1, \nabla^\Gamma h_1
$
without arguments are evaluated at $\rho, (\rho,s,t)$, $(s,t)$, $(r,\rho,s,t)$, $(r,s,t)$, respectively.
 Moreover, we have used $-\theta_0''(\rho)+ f'(\theta_0(\rho))=0$ for all $\rho\in\R$ as well as $\eps (V-\ve\cdot \no)\p_\rho \tc_2\in \mathcal{R}_{1,\alpha}^0$ and $\eps (|\nabla^\Gamma h_1|^2-|\nabla_\Gamma h_1|^2)\theta_0''$ since $V(s,t)-\ve(x,t)\cdot \no(s,t)=0$ and $|\nabla^\Gamma h_1|^2-|\nabla_\Gamma h_1|^2=0$ if $x=X_0(s,t)$.

First we eliminate the terms of order $O(\eps^{-1})$ on the right hand side of \eqref{eq:monst1}.
Using a Taylor expansion around $r=d_\Gamma=0$ and $d_\Gamma=\eps (\rho+h_\eps)$ we have
  \begin{equation}\label{eq:ex2}
    -V(s,t)+\no\cdot X_0^\ast(\ve)(s,t)  = \eps(\rho+h_\eps)\kappa_1(s,t)+ \eps^2 \kappa_2(s,t)(\rho+h_1)^2+\eps^3 \kappa_{3,\eps}(\rho,s,t)
  \end{equation}
  for some smooth and uniformly bounded $\kappa_1, \kappa_2, \kappa_{3,\eps}$.
 Therefore we arrive at
\begin{alignat}{1}\nonumber
&\p_t c^{in}+\ve\cdot\nabla c^{in}-\eps\Delta c^{in} +\frac1{\eps} f'(c^{in})\\\nonumber
  &=\left(h_1\kappa_1  -\p_t^\Gamma h_1-\ve\cdot \nabla^\Gamma h_1-\Delta d_\Gamma \right)\theta_0'\\\nonumber
  &\ -\p_\rho^2 \tc_1+f''(\theta_0) \tc_1-|\nabla_\Gamma h_1|^2\theta_0''+ \rho \theta_0'(\rho)\kappa_1 + \eps (-\p_\rho^2 \tc_2+f''(\theta_0) \tc_2)+\eps(\rho + h_1\kappa_1)\p_\rho \tc_1\\
  &\ +\eps(-\p_t^\Gamma h_2-\ve\cdot \nabla^\Gamma h_2+  \kappa_1 h_2+\Delta^\Gamma h_1+ \kappa_2(\rho+h_1)^2)\theta_0' 
 \label{eq:monst1'}
 + R_1^0 + R_2
\end{alignat}
in $\Gamma(2\delta)$ for some $R_1^0 \in \mathcal{R}_{1,\alpha}^0$, $R_2\in \mathcal{R}_{2,\alpha}$.
Now, we want to eliminate all terms of order $O(1)$.
To this end we choose $h_1$ as the solution of
  \begin{equation}\label{eq:h1equ}
       D_t  h_1 +X_0^\ast(\ve) \cdot  \nabla_\Gamma h_1    -\kappa_1 h_1 = \Delta_\Gamma d_\Gamma\quad \text{on }\mathbb{T}^1\times [0,T_0]
   \end{equation}
   together with  $h_1|_{t=0}=0$. Then using a Taylor expansion similarly as before
   \begin{equation*}
     \left(h_1\kappa_1  -\p_t^\Gamma h_1(r,s,t)-\ve\cdot \nabla^\Gamma h_1-\Delta d_\Gamma \right)\theta_0'
     = \eps (\rho+h_1)b(s,t)\theta_0'(\rho) + R_1^0
\end{equation*}
for some $R_1^0 \in \mathcal{R}_{1,\alpha}^0$ and some smooth $b$.
This changes \eqref{eq:monst1'} into
\begin{alignat}{1}\nonumber
&\p_t c^{in}+\ve\cdot\nabla c^{in}-\eps\Delta c^{in} +\frac1{\eps} f'(c^{in})\\\nonumber
  &= -\p_\rho^2 \tc_1+f''(\theta_0) \tc_1-|\nabla_\Gamma h_1|^2\theta_0''+ \rho \theta_0'(\rho)\kappa_1 +\eps\left(-\p_\rho^2\tc_2f''(\theta_0)\tc_2\right)+\eps(\rho + h_1{\kappa_1)}\p_\rho \tc_1\\
  &\ +\eps\left(-\p_t^\Gamma h_2-\ve\cdot \nabla^\Gamma h_2+  \kappa_1 h_2+\Delta^\Gamma h_1+ \kappa_2(\rho+h_1)^2+(\rho+h_1)b\right)\theta_0'
 + R_1^0 + R_2  \label{eq:monst2}
\end{alignat}
in $\Gamma(2\delta)$ for some $R_1^0 \in \mathcal{R}_{1,\alpha}^0$, $R_2\in \mathcal{R}_{2,\alpha}$.

In order to eliminate the remaining $O(1)$ terms, we only need to choose $\tc_1$ such that
\begin{equation*}
  -\p_\rho^2 \tc_1+f''(\theta_0) \tc_1=|\nabla_\Gamma h_1|^2\theta_0''(\rho)- \theta_0'(\rho)   \rho \kappa_1
\end{equation*}
for all $(\rho,s,t)\in\R\times \T^1\times [0,T_0]$. The solvability is guaranteed by:
\begin{prop}\label{prop:solveode}
  Assume that $g\colon \R\times\T^1\times [0,T_0]\rightarrow\R$ and $g^\pm\colon \T^1\times [0,T_0]\rightarrow\R $ are smooth 
 and for some $i\in\N_0$, $a>0$ satisfy
  \begin{equation*}
    \sup_{(s,t)\in \T^1\times [0,T_0]}\left|\p_\rho^k\partial_s^l \p_t^m[g(\rho,s,t)-g^\pm(s,t)]\right|\leq C_{k,l,m}(1+|\rho|)^ie^{-a|\rho|}\quad \text{for all } \rho\gtrless 0.
  \end{equation*}
  for all $k,l,m\in \mathbb{N}_0$ 
  and some $C_{k,l,m}>0$. Then for given $(s,t)\in \T^1\times [0,T_0]$ the ODE
  \begin{equation*}
    -\p^2_\rho u+f''(\theta_0)u=g(\cdot ,s,t)~\text{in}~\mathbb{R},\quad u(0,s,t)=0,
  \end{equation*}
  has a unique bounded solution $u(\cdot,s,t)$ if and only if
  \begin{equation}\label{lem:CompCondODE}
    ~\int_\mathbb{R}g(\rho,s,t)\theta'_0(\rho) \sd \rho=0. 
  \end{equation}
  If the solution exists for all $(s,t)\in \T^1\times [0,T_0]$, then for all $(k,l,m)\in \mathbb{N}_0^3$ 
  there is some $C_{k,l,m}$ such that
   \begin{equation}\label{eq:matching1}
    \sup_{(s,t)\in \T^1\times [0,T_0]}\left|\p_\rho^k\partial_s^l \p_t^m\left(u(\rho,s,t)-\frac{g^\pm(s,t)}{f''(\pm 1)}\right)\right|\leq C_{k,l,m}(1+|\rho|)^ie^{-a|\rho|}\quad\text{for all }\rho\gtrless 0.
  \end{equation}
\end{prop}
\begin{proof}
  See e.g.\ \cite[Lemma 3]{ChenHilhorstLogak}.
\end{proof}

We note that in the present application to $\tc_1$ and $\tc_2$ below $g^\pm \equiv 0$. Thus we obtain for all $k,l,m\in\N_0$ that there is some $C_{k,l,m}>0$ such that
\begin{equation}\label{eq:Estimc1}
\sup_{(s,t)\in \T^1\times [0,T_0]}\left|\p_\rho^k\partial_s^l \p_t^m\tc_1(\rho,s,t)\right|\leq C_{k,l,m}(1+|\rho|)e^{-\alpha|\rho|}\quad\text{for all }\rho\in\R,
\end{equation}
where $\alpha$ is as in \eqref{eq:decayopti}.
In particular we have exponential decay as $|\rho|\to \infty$.

With this choice  \eqref{eq:monst2} simplifies to
\begin{alignat}{1}\nonumber
&\p_t c^{in}+\ve\cdot\nabla c^{in}-\eps\Delta c^{in} +\frac1{\eps} f'(c^{in})\\\nonumber
&=  \eps\left(\Delta^\Gamma h_1-\p_t^\Gamma h_2-\ve\cdot \nabla^\Gamma h_2+  \kappa_1 h_2- g\right)\theta_0'\\
&\quad +\eps\left(  (g+\kappa_2(\rho+h_1)^2+(\rho+h_1)b)\theta_0'+(\rho + h_1)\kappa_1\p_\rho \tc_1-\p_\rho^2\tc_2 +f''(\theta_0)\tc_2\right) 
\\&\quad + R_1^0 + R_2 \qquad \text{in }\Gamma(2\delta)
\end{alignat}
for some $R_1^0 \in \mathcal{R}_{1,\alpha}^0$, $R_2\in \mathcal{R}_{2,\alpha}$.
Here  we added a term $g=g(s,t)$ in order to satisfy the compatibility condition for the equation for $\tc_2$ in the sequel. 
Now we choose  $h_2$ as the solution of
\begin{alignat}{1}\label{eq:h2equ}
   & D_t  h_2  +   X_0^\ast(\ve)\cdot \nabla_\Gamma h_2 - \kappa_1h_2   =   \Delta_\Gamma h_1-g- \kappa_1h_2
\end{alignat}
on $\T^1\times [0,T_0]$
together with the initial condition $h_2|_{t=0}=0$.
Existence of a solution follows from Proposition~\ref{prop:solveode}.

It remains to eliminate all the $O(\eps)$ terms by solving the following equation for $\tc_2$:
\begin{alignat}{1}
&-\p_\rho^2\tc_2+f''(\theta_0)\tc_2=-(g+\kappa_2(\rho+h_1)^2+(\rho+h_1)b)\theta_0'-(\rho + h_1)\kappa_1\p_\rho \tc_1{.} \label{eq:monst4}
\end{alignat}
This system is solvable if and only we choose $g$ such that
\begin{equation}\label{eq:deterb}
g(s,t)\int_\R(\theta_0')^2\sd \rho
  =-\int_\R\left[(\kappa_2(\rho+h_1)^2+(\rho+h_1)b)\theta_0'+(\rho + h_1)\kappa_1\p_\rho \tc_1   \right]\theta_0'(\rho)\sd \rho{.}
\end{equation}
With this choice of $g$ we obtain by  Proposition~\ref{prop:solveode} the existence of $\tc_2$, which satisfies for arbitrary $k,l,m\in\N_0$ the estimate
\begin{equation}\label{eq:Estimc2}
\sup_{(s,t)\in \T^1\times [0,T_0]}\left|\p_\rho^k\partial_s^l \p_t^m\tc_2(\rho,s,t)\right|\leq C_{k,l,m}(1+|\rho|^2)e^{-a|\rho|}\quad\text{for all }\rho\in\R
\end{equation}
for some $C_{k,l,m}>0$.

Finally, we arrive at
\begin{alignat}{1}\label{eq:ApproxS3}
&\p_t c^{in}+\ve\cdot\nabla c^{in}-\eps\Delta c^{in} +\frac1{\eps} f'(c^{in})=  R_1^0 + R_2\quad \text{in }\Gamma(2\delta)
\end{alignat}
 for some $R_1^0 \in \mathcal{R}_{1,\alpha}^0$, $R_2\in \mathcal{R}_{2,\alpha}$. Altogether we obtain
\begin{thm}
  Let $c_A$ be constructed as before. Then
  \begin{alignat}{1}\label{eq:ApproxS4}
    &\p_t c_A+\ve\cdot\nabla c_A-\eps\Delta c_A +\frac1{\eps} f'(c_A)=  S_\eps\quad \text{in }\Omega \times [0,T_0],
  \end{alignat}
  where
  \begin{equation*}
    \|S_\eps\|_{L^2(\Omega\times (0,T_0))}\leq C\eps^{2+\frac12}\qquad \text{for all }\eps \in (0,1].
  \end{equation*}
\end{thm}
\begin{proof}
  Because of the definition \eqref{eq:cA}
  \begin{align*}
    &\p_t c_A+\ve\cdot\nabla c_A-\eps\Delta c_A +\frac1{\eps} f'(c_A)\\
    &\quad = \zeta(d_\Gamma)\left(\p_t c^{in}+\ve\cdot\nabla c^{in}-\eps\Delta c^{in} +\frac1{\eps} f'(c^{in})\right)\\
    &\qquad  + \zeta'(d_\Gamma)\left(\p_t d_\Gamma+\ve\cdot \nabla d_\Gamma -\eps \Delta d_\Gamma\right)(c^{in} -\chi_{\Omega^+}+\chi_{\Omega^-})\\
    &\qquad - \eps\zeta''(d_\Gamma)|\nabla d_\Gamma|^2(c^{in} -\chi_{\Omega^+}+\chi_{\Omega^-})+ \frac1\eps \left(f'(c_A)-\zeta(d_\Gamma)f'(c_A^{in})\right) =: S_\eps.
  \end{align*}
  Here $\zeta'(d_\Gamma)$ and $\zeta''(d_\Gamma)$ are supported in $\ol{\Gamma(2\delta)}\setminus \Gamma(\delta)$, where $|\rho|\geq \frac{c_0}\eps$ for some $c_0>0$. Hence using
  \begin{equation*}
    |\tc^{in}(\rho,s,t)\mp 1|\leq Ce^{-\alpha |\rho|}\qquad \text{for all }\rho \gtrless 0
  \end{equation*}
  because of the properties of $\theta_0$, \eqref{eq:Estimc1}, and \eqref{eq:Estimc2}, we obtain that the terms with  $\zeta'(d_\Gamma)$ and $\zeta''(d_\Gamma)$ decay exponentially as $\eps\to 0$ (in any $C^k$-norm). Moreover, we note that $f'(c_A)-\zeta(d_\Gamma)f'(c_A^{in})$ is supported in $\ol{\Gamma(2\delta)}\setminus \Gamma(\delta)$ since $f'(\pm 1)=0$. Moreover, using \eqref{eq:decayopti} one can show that this term decays exponentially as $\eps\to 0$ as well.
  Now using \eqref{eq:ApproxS3} together with Corollary~\ref{cor:repsEstim2}, we conclude
\begin{equation*}
  \|S_\eps\|_{L^2(\Omega\times (0,T_0))}\leq C\eps^{2+\frac12}\qquad \text{for all }\eps \in (0,1].
\end{equation*}
\end{proof}

For the following let us denote by
\begin{alignat*}{2}
    c_{A,0}&=\zeta\circ d_\Gamma \theta_0(\rho)+(1-\zeta\circ d_\Gamma)\left(
  \chi_{\Omega^+}-\chi_{\Omega^-} \right)&\quad& \text{in } \Om\times [0,T_0]
\end{alignat*}
the leading part of $c_A$. Then we have the following spectral estimate.
\begin{thm}\label{thm:Spectral}
Let $c_{A,0}$ be as above. Then there are some $C,\eps_0>0$  such that for every $u\in H^1(\Om)$, $t\in[0,T_0]$, and $\eps\in (0,\eps_0]$ we have
  \begin{align*}
    &\int_{\Om}\left(|\nabla u(x)|^2+ \eps^{-2}f''(c_{A,0}(x,t))u(x)^2\right)\sd x\\
    &\quad \geq -C\int_\Om u(x)^2 \sd x + \int_{\Om\setminus \Gamma_t(\delta)} |\nabla u(x)|^2\sd x +  \int_{\Gamma_t({\delta})} |\nabla_\btau u(x)|^2\sd x.
  \end{align*}
\end{thm}
\begin{proof}
  This is a special case of \cite[Theorem~2.13]{StokesAllenCahn}.
\end{proof}

In the present case of a vanishing mobility $m_\eps = m_0 \eps$ the following weaker version of the spectral estimate (with $c_A$ instead of $c_{A,0}$) is sufficient. 
\begin{cor}\label{cor:Spectral}
  Let $c_A$ be as before. Then there are $C_L,\eps_0>0$  such that for every $u\in H^1(\Om)$, $t\in[0,T_0]$, and $\eps\in (0,\eps_0]$ we have
  \begin{equation*}
    \int_{\Om}\left(\eps|\nabla u|^2+ \eps^{-1}f''(c_A(\cdot ,t))u^2\right)\sd x\geq -C_L\int_\Om u^2 \sd x + {\eps}\int_{\Om\setminus \Gamma_t(\delta)} |\nabla u|^2\sd x +  {\eps}\int_{\Gamma_t({\delta})} |\nabla_\btau u|^2\sd x.
  \end{equation*}
\end{cor}
\begin{proof}
  By the construction $c_A$ and $c_{A,0}$ are uniformly bound with respect to  $x\in \ol{\Om}$, $t\in [0,T_0]$ and $\eps \in (0,1]$. Hence by a Taylor expansion
  \begin{align*}
    f''(c_A(x,t))&= f''(c_{A,0}(x,t))+ f'''( \xi c_{A,0}(x,t)+ (1-\xi) c_A(x,t))(c_A(x,t)-c_{A,0}(x,t))\\
    &\geq f''(c_{A,0}(x,t)) - c_0 \eps\quad \text{for all } x\in \ol{\Om}, t\in [0,T_0],\eps \in (0,1]
  \end{align*}
  for some $c_0>0$, where $\xi=\xi(x,t,\eps)\in [0,1]$. Thus we derive
  \begin{align*}
    &\int_{\Om}\left(\eps|\nabla u|^2+ \eps^{-1}f''(c_A(\cdot ,t))u^2\right)\sd x\geq \int_{\Om}\left(\eps|\nabla u|^2+ \eps^{-1}f''(c_{A,0}(\cdot ,t))u^2\right)\sd x - c_0 \int_\Omega u^2\, dx\\
    &\qquad \geq - (\eps_0 C +c_0)\int_\Om u(x)^2 \sd x + {\eps}\int_{\Om\setminus \Gamma_t(\delta)} |\nabla u(x)|^2\sd x +  {\eps}\int_{\Gamma_t({\delta})} |\nabla_\btau u(x)|^2\sd x
  \end{align*}
  for all $u\in H^1(\Omega)$ and $t\in[0,T_0]$, $\eps \in (0,\eps_0]$ by Theorem~\ref{thm:Spectral}. This yields the statement with $C_L=\eps_0 C+c_0$.
\end{proof}

\section{Error Estimates}\label{sec:Main}

The proof of our main result Theorem~\ref{thm:main} follows the same steps as in \cite[Section~4]{StokesAllenCahn}. 
Let $u_\eps:= c_\eps-c_A$ denote the error of the approximate solution $c_A$. We consider the estimates
\begin{subequations}\label{assumptions}
  \begin{align}\label{assumptions1}
    \sup_{0\leq t\leq \tau} \|u_\eps(t)\|_{L^2(\Omega)} + {\eps^{\frac{1}{2}}}\|\nabla u_\eps\|_{L^2(\Omega\times(0,\tau)\setminus \Gamma(\delta))}&+{\eps^{\frac{1}{2}}}\|\nabla_{{\btau}}u_\eps \|_{L^2(\Omega\times(0,\tau)\cap \Gamma(\delta))} \leq R\eps^{\order+\frac12},\\\label{assumptions2}
     \eps\|\partial_{\no}u_\eps \|_{L^2(\Omega\times(0,\tau)\cap \Gamma(2\delta))} &\leq R\sqrt{C_fT_0}\eps^{\order+\frac12}
\end{align}
\end{subequations}
for some $\tau=\tau(\eps) \in (0,T_0]$, $\eps_0\in (0,1]$, and all $\eps \in (0,\eps_0]$, where
\begin{equation*}
C_f:= \sup_{\eps \in (0,1), x\in\Omega, t\in [0,T_0]} |f''(c_A(x,t))|.  
\end{equation*}

We choose $R>0$ so large that
\begin{equation}
  \sup_{0<\eps\leq \eps_0} e^{(2C_L+1)T_0} \|S_\eps\|_{L^2(\Omega\times (0,T_0))}^2 \leq \frac{R^2}4 \eps^{2N+1}
\end{equation}
for all $\eps \in (0,\eps_0]$, where $S_\eps$ is as in Corollary~\ref{cor:Spectral}, and
 \begin{equation}\label{initial assumption-0}
e^{(2C_L+1)T_0}  \|c_{0,\eps}-c_A|_{t=0}\|_{L^2(\Omega)}^2\leq \frac{R^2}{4}\eps^{2\order+1}
\end{equation}
for all $\eps\in (0,1]$, where $C_L>0$ is the constant from the spectral estimate in Corollary~\ref{cor:Spectral}. Let us remark that compared to \cite[Estimates (5.5)]{StokesAllenCahn} there is an additional factor $\eps^{\frac12}$ in front of the norms for $\nabla (c_\eps-c_A)$  and $\nabla_{\btau} (c_\eps-c_A)$.

As in \cite[Section~5.2]{StokesAllenCahn}, we use a continuation argument to show \eqref{assumptions} for $\tau=T_0$ if $\eps_0>0$ is sufficiently small. To this end we define
\begin{align}
T_{\varepsilon}:=\sup\{\tau\in[0,T_0]: \eqref{assumptions}  \ \text{holds true}\}.\label{def:Teps}
\end{align}
Then we have $T_{\varepsilon}>0$ because of \eqref{initial assumption-0}.

In order to estimate the error due to linearization of $\frac1\eps f'(c)$ we need:
\begin{prop}\label{prop:u3Estim}
  Under the assumptions before we have for every $T\in (0,T_\eps)$
\begin{align}
\int_0^{T}\int_{\Omega} |u_\eps|^3\sd x\sd t\leq C(R)T^{\frac{1}{2}}\varepsilon^{3N+\frac{3}{4}}.\label{est:nonlinear}
 \end{align}
\end{prop}
\begin{proof}
  The proof is a modification of the proof of \cite[Proposition~4.3]{AbelsFei}. The only difference is that the power of $\eps$ in the estimate for $\|\nabla_{\btau} u_\eps\|_{L^2}^{\frac12}$ is decreased by $\frac14$, which causes the loss of $\frac14$ in the power of $\eps$ compared to \cite[Proposition~4.3]{AbelsFei}.
\end{proof}

The equations \eqref{ConvAC1} and \eqref{eq:ApproxS4} imply
 \begin{alignat}{2}
   \partial_t u_\eps&+\ve\cdot \nabla u_\eps\label{eq:u}
   = \eps\Delta u_\eps-\frac1\eps f''(c_A)u_\eps -\frac1\eps \mathcal{N}(c_A,u_\eps)- S_\eps,
  \end{alignat}
where
$
  \mathcal{N}(c_A,u_\eps)=f'(c_\eps)-f'(c_A)-f''(c_A)u_\eps.
$
 Taking the $L^2(\Omega)$-inner product of \eqref{eq:u} and $u_\eps$ and using integration by parts we obtain
  \begin{alignat}{2}\nonumber
    &\frac{1}{2}\frac{d}{dt}\int_{\Omega} u_\eps^2\sd x+ \eps\int_{\Omega}\big( |\nabla u_\eps|^2+\frac{f''(c_A)}{\eps^2}u_\eps^2\big) \sd x\\\label{ineq:u energy-00}
    &\quad
    \leq \frac{C}\eps\int_{\Omega}|u_\eps|^3\sd x+\frac12\|S_\eps\|_{L^2(\Omega)}^2+ \frac12\|u_\eps\|_{L^2(\Omega)}^2
  \end{alignat}
  because of
  \begin{align*}
  \int_{\Omega} \mathcal{N}(c_A,u_\eps)u_\eps\sd x\geq-C\int_{\Omega} |u_\eps|^3\sd x.
  \end{align*}
  Now the spectral estimate due to Corollary~\ref{cor:Spectral} implies
  \begin{alignat}{2}\label{ineq:u energy}
    &\frac{1}{2}\frac{d}{d t}\int_{\Omega} u_\eps^2\sd x-C_L\int_\Om u_\eps^2 \sd x +\eps \int_{\Om\setminus \Gamma_t(\delta)} |\nabla u_\eps|^2\sd x + \eps \int_{\Gamma_t(\delta)} |\nabla_{\btau} u_\eps|^2\sd x\nonumber\\&\quad \leq
    \frac{C}\eps\int_{\Omega}|u_\eps|^3\sd x+\frac12\|S_\eps\|_{L^2(\Omega)}^2+ \frac12\|u_\eps\|_{L^2(\Omega)}^2.
  \end{alignat}
Hence we derive by Gronwall's inequality 
  \begin{alignat}{2}\label{ineq:u energy-1}
    &\sup_{0\leq t\leq T}\int_{\Omega} u_\eps(x,t)^2\sd x + 2\eps\int_{0}^T\int_{\Om\setminus \Gamma_t(\delta)} |\nabla u_\eps|^2\sd x \sd t +  2\eps\int_{0}^T\int_{\Gamma_t(\delta)} |\nabla_{\btau} u_\eps|^2d x \sd t\nonumber\\&\leq e^{(2C_L+1)T_0}\left(\int_{\Omega} u_\eps(x,0)^2\sd x +\frac{C}\eps\int_{0}^T\int_{\Omega}|u_\eps|^3\sd x\sd t+\|S_\eps\|_{L^2(\Omega\times (0,T_0))}^2\right)
    \nonumber\\&\leq \frac{R^2}{2}\eps^{2N+1}+C\eps^{3N-\frac14}
  \end{alignat}
  for all $0\leq T\leq T_\varepsilon$.
  Thus, if $\varepsilon\in(0,\varepsilon_0)$ and $\varepsilon_0>0$ is sufficiently small, we have
  \begin{align*}
\frac{R^2}{2}\eps^{2N+1}+C\eps^{3N-\frac14}\leq \frac{2R^2}{3}\eps^{2N+1}\quad \text{for all }\eps\in (0,\eps_0]
  \end{align*}
  since $N>\frac54$
and therefore
 \begin{alignat}{2}\nonumber
   &\sup_{0\leq t\leq T}\int_{\Omega} u_\eps(x,t)^2\sd x +2\eps \int_{\Om\times (0,T)\setminus \Gamma(\delta)} |\nabla u_\eps|^2\sd (x,t) \\\label{ineq:u energy-2}
   &\quad+ 2\eps \int_{\Omega\times (0,T)\cap \Gamma(\delta)} |\nabla_{\btau} u_\eps|^2\sd (x,t)\leq \frac{2R^2}{3}\eps^{2N+1}< R^2\eps^{2N+1}.
  \end{alignat}
Hence together with  \eqref{ineq:u energy-00} we derive
  \begin{alignat}{2}
    &\eps\int_0^T \int_{\Omega}\left( |\nabla u_\eps|^2+\frac{f''(c_A)}{\eps^2}u_\eps^2\right) \sd x \sd t\leq\frac{1}{2}\int_{\Omega} u_\eps(x,0)^2\sd x
     \nonumber\\&\quad+ \frac{C}\eps\int_{0}^T\int_{\Omega}|u_\eps|^3\sd x\sd t+\int_0^{T_0}\|S_\eps\|_{L^2(\Omega)}^2\, dt+\int_0^{T_0} \|u_\eps\|_{L^2(\Omega)}^2\, dt
    \nonumber\\\label{ineq:u energy-4}
    &\quad\leq C(R,T_0)\eps^{2N+1}
  \end{alignat}
   and, if $\eps_0>0$ is sufficiently small,
  \begin{alignat}{2}\label{ineq:u energy-5}
   \eps^{2}\int_0^T \int_{\Omega}|\partial_{\nn} u_\eps|^2\sd x \sd t&\leq\eps^{2}\int_0^T \int_{\Omega}|\nabla u_\eps|^2\sd x \sd t
    \nonumber\\&\leq\int_0^T \int_{\Omega}\big( \eps^{2}|\nabla u_\eps|^2+f''(c_A)u^2\big) \sd x \sd t+C_f\int_0^T \int_{\Omega}u_\eps^2\sd x \sd t\nonumber\\&\leq C(R,T_0)\eps^{2N+2}+C_fT_0\frac{2R^2}{3}\eps^{2N+1}
    < C_f T_0 R^2\eps^{2N+1}.
  \end{alignat}
  Since the inequalities in \eqref{ineq:u energy-2} and \eqref{ineq:u energy-5} are strict and the left-hand sides in \eqref{assumptions} are continuous with respect to $\tau$, we conclude $T_\eps =T_0$ for all $\eps\in (0,\eps_0]$ by the definition of $T_\eps$ provided $\eps_0$ is sufficiently small.

  \section*{Acknowledgements}
 {We are grateful to the anonymous referee for the careful reading and useful comments.} This work was supported by the Research Institute for Mathematical Sciences, an
  International Joint Usage/Research Center located in Kyoto University.


  \def\ocirc#1{\ifmmode\setbox0=\hbox{$#1$}\dimen0=\ht0 \advance\dimen0
  by1pt\rlap{\hbox to\wd0{\hss\raise\dimen0
  \hbox{\hskip.2em$\scriptscriptstyle\circ$}\hss}}#1\else {\accent"17 #1}\fi}

\bigskip

\noindent
{\it
  (H. Abels) Fakult\"at f\"ur Mathematik,
  Universit\"at Regensburg,
  93040 Regensburg,
  Germany}\\
{\it E-mail address: {\sf helmut.abels@mathematik.uni-regensburg.de} }\\[1ex]

\end{document}